\documentclass[12pt]{amsart}

\usepackage{amssymb,amsmath,amsfonts, amsthm}
\usepackage{graphicx}

\newcommand{\R}{\mathbb{R}}
\newcommand{\N}{\mathbb{N}}
\newcommand{\Z}{\mathbb{Z}}
\newcommand{\Q}{\mathbb{Q}}
\newcommand{\C}{\mathbb{C}}
\newcommand{\bbH}{\mathbb{H}}

\newcommand{\Diag}{\mathrm{Diag}}
\newcommand{\tr}{\mathrm{tr}}
\newcommand{\arcosh}{\mathrm{arcosh}}
\newcommand{\Tr}{\mathrm{Tr}}

\newcommand{\GL}{\mathrm{GL}}
\newcommand{\SL}{\mathrm{SL}}
\newcommand{\PSL}{\mathrm{PSL}}

\newcommand{\Ha}{\mathrm{\textbf{H}}}
\newcommand{\RP}{\mathbb{RP}}
\newcommand{\Lim}{\mathcal{L}}

\newtheorem{theorem}{Theorem}
\newtheorem{proposition}[theorem]{Proposition}
\newtheorem{corollary}[theorem]{Corollary}
\newtheorem{lemma}[theorem]{Lemma}

\newtheorem*{prop}{Proposition}
\newtheorem*{thm}{Theorem}

\theoremstyle{definition}

\theoremstyle{remark}
\newtheorem*{rem}{Remark}

\title{Examples of infinite covolume subgroups of $\PSL(2,\R)^r$ with big limit sets.}
\author{Slavyana Geninska}
\address{ Karlsruher Institut f\"ur Technologie, Kaiserstr. 89-93, 76133 Karlsruhe, Germany}
\email{slavyana.geninska@kit.edu}
\address{Universit\'e de Provence, 39 rue F. Joliot Curie, 13453 Marseille, France}
\email{geninska@latp.univ-mrs.fr}

\begin{document}

\begin{abstract}
We provide examples of finitely generated infinite covolume subgroups of $\PSL(2,\R)^r$ with a "big" limit set, e.g. that contains an open subset of the geometric boundary. They are given by the so called semi-arithmetic Fuchsian groups admitting modular embeddings.
\end{abstract}

\maketitle

\pagestyle{plain}
\setcounter{page}{1}

\setcounter{section}{-1}
\section{Introduction}
While lattices are studied very well, only little is known about discrete subgroups of infinite covolume of semi-simple Lie groups. The main class of examples are Schottky groups. In this paper we consider another class of examples for the semi-simple Lie group $\PSL(2,\R)^r$ with $r\geq 2$.

We provide examples of finitely generated infinite covolume subgroups of $\PSL(2,\R)^r$ with a "big" limit set, e.g. that contains an open subset of the geometric boundary. They are given by the so called semi-arithmetic Fuchsian groups admitting modular embeddings.

The semi-arithmetic Fuchsian groups constitute a specific class of Fuchsian groups which can be embedded up to commensurability in arithmetic subgroups of $\PSL(2,\R)^r$ (see Schmutz Schaller and Wolfart~\cite{pS00}). These embeddings are of infinite covolume in $\PSL(2,\R)^r$. A trivial example is the group $\PSL(2,\Z)$ that can be embedded diagonally in any Hilbert modular group. Further examples are the other arithmetic Fuchsian groups, the triangle Fuchsian groups and the Veech groups of Veech surfaces.

It is a general question if certain classes of Fuchsian groups can be characterized by geometric means. In \cite{sG10a} was shown that a nonelementary finitely generated subgroup of arithmetic groups in $\PSL(2,\R)^r$ with $r\geq 2$ has the smallest possible limit set if and only if its projection to one factor is a subgroup of an arithmetic Fuchsian group. This means that in this case the limit set gives information about the arithmetic properties of the group. In this article we show that the groups admitting modular embeddings that are not arithmetic are in a sense really very different to the arithmetic Fuchsian groups. It is still an open question what are the limit sets of the embeddings of the other semi-arithmetic Fuchsian groups in  arithmetic subgroups of $\PSL(2,\R)^r$.

\bigskip

This article is organized as follows. Section 1 provides a detailed description of the geometric boundary of $(\bbH^2)^r$, which is the set of equivalence classes of asymptotic geodesic rays. We introduce the notion of the limit set as the part of the orbit closure $\overline{\Gamma(x)}$ in the geometric boundary where $x$ is an arbitrary point in $(\bbH^3)^q\times(\bbH^2)^r$. We also state a natural structure theorem for the regular limit set $\Lim_\Gamma^{reg}$ of discrete nonelementary groups $\Gamma$ due to Link: $\Lim_\Gamma^{reg}$ is the product of the Furstenberg limit set $F_\Gamma$ and the projective limit set $P_\Gamma$.

In Section 2 we prove the following criterion for Zariski dense subgroups of arithmetic groups in $\PSL(2,\R)^r$ (Corollary~\ref{C:CritZariski} in the text). For a group $S$ we denote by $S^{(2)}$ its subgroup generated by the set $\{g^2\mid g\in S\}$.

\begin{prop}
Let $\Gamma$ be a nonelementary subgroup of an irreducible arithmetic group $\Delta$ in $\PSL(2,\R)^r$. Then $\Gamma$ is Zariski dense in $\PSL(2,\R)^r$ if and only if the fields generated by $\Tr(p_i(\Gamma^{(2)}))$ and $\Tr(p_i(\Delta^{(2)}))$ are equal for one and hence for all $i \in \{1,\ldots,r\}$, where $p_i$ denotes the projection to the $i$-th factor.
\end{prop}



In Section 3, we provide examples of ``small" groups with a ``big" limit set, namely groups of infinite covolume for which the projection of the regular limit set into the Furstenberg boundary is the whole Furstenberg boundary. We define semi-arithmetic groups admitting a modular embeding. A cofinite Fuchsian group $S$ that is commensurable to a subgroup of the projection to the first factor of an irreducible arithmetic group $\Delta$ in $\PSL(2,\R)^r$ is said to have a \textit{modular embedding} if  for the natural embedding $f:S \rightarrow \Delta$ there exists a holomorphic embedding $F:\bbH^2 \rightarrow (\bbH^2)^r$ with
$$ F(T z) =  f(T)F(z), \quad \text{for all} \quad T\in S \quad \text{and all} \quad z \in \bbH^2.$$
Examples of semi-arithmetic groups admitting modular embeddings are Fuchsian triangle groups (see Cohen and Wolfart \cite{pC90}).

The next theorem is the main result of this paper. It is a compilation of Lemma~\ref{L:MinPower}, Theorem~\ref{T:BigFurstenberg} and Theorem~\ref{T:BigLimitSet} in the text.

\begin{thm}
Let $\Gamma$ be a subgroup of an irreducible arithmetic group in $\PSL(2,\R)^r$ with $r\geq 2$ such that $p_j(\Gamma)$ is a semi-arithmetic Fuchsian group admitting a modular embedding and $r$ is the smallest power for which $p_j(\Gamma)$ has a modular embedding in an irreducible arithmetic subgroup of $\PSL(2,\R)^r$.  Then 
\begin{itemize}
\item[(i)] the Furstenberg limit set $F_\Gamma$ is the whole Furstenberg boundary $(\partial \bbH^2)^r$,
\item[(ii)] the limit set $\Lim_\Gamma$ contains a subset homeomorphic to $D^{2r-1}$ (the $(2r-1)$-dimensional ball).
\end{itemize}
\end{thm}

In Section 4 we consider the case with parabolic elements. In the remark after Proposition~\ref{P:Parabolic} we determine the exact limit set of the embeddings in $\PSL(2,\R)^2$ of the triangle groups with signature $(2,5,\infty)$ and $(5,\infty,\infty)$.

I would like to thank the advisor of my thesis Enrico Leuzinger. I would also like to thank Yves Benoist and George Tomanov for the discussion about the algebraic subgroups of $\SL(2,\R)^r$. 

\section{Background}
\setcounter{theorem}{0}
In this section we provide some basic facts and notations that are needed in the rest of this paper. 

We will change freely between matrices in $\SL(2,\R)$ and their action as fractional linear transformations, namely as elements in $\PSL(2,\R)$.

For $g = \begin{bmatrix} a&b\\c&d \end{bmatrix} \in \PSL(2,\R)$ we set $\tr(g)=|a+d|$. 

For a subgroup $\Gamma$ of $\PSL(2,\R)$ we call 
\[ \Tr(\Gamma) = \{\tr(g) \mid g \in \Gamma \} \]
the \textit{trace set of} $\Gamma$. 

The \textit{translation length} $\ell(g)$ of a hyperbolic isometry $g$ is the distance between a point $x$ on the geodesic fixed by $g$ and its image $g(x)$ under $g$. If $g$ is elliptic, parabolic or the identity, we define $\ell(g):=0$.

The following notion of ``smallness" for subgroups $\Gamma$ of $\PSL(2,\R)$ is important in the subsequent discussion. The group $\Gamma$ is \textit{elementary} if it has a finite orbit in its action on $\bbH^2\cup\R\cup\{\infty\}$. Otherwise it is said to be \textit{nonelementary}. Every nonelementary subgroup of $\PSL(2,\R)$ contains infinitely many hyperbolic elements, no two of which have a common fixed point (see Theorem 5.1.3 in the book of Beardon~\cite{aB95}).

A Schottky group is a finitely generated free subgroup of $\PSL(2,\R)$ that contains only hyperbolic isometries except for the identity. We will mainly deal with two-generated Schottky groups.

For each two hyperbolic isometries without common fixed points, we can find powers of them that generate a Schottky group. This means that every nonelementary subgroup of $\PSL(2,\R)$ has a subgroup that is a Schottky group. A proof of this fact can be found in \cite{sG09}.

A Schottky group contains isometries without common fixed points because it is nonelementary.

\subsection{The geometric boundary of $(\bbH^2)^r$}
For $i=1,\ldots,r$, we denote by $p_i: (\bbH^2)^r \rightarrow \bbH^2$, $(z_1,...,z_{r}) \mapsto z_i$ the $i$-th projection of $(\bbH^2)^r$ into $\bbH^2$. The curve $\gamma: [0, \infty )\rightarrow (\bbH^2)^r$ is a geodesic ray in $(\bbH^2)^r$ if and only if $p_i\circ \gamma$ is a geodesic ray or a point in $\bbH^2$ for each $i = 1, \ldots , r$. A geodesic $\gamma$ is \textit{regular} if $p_i\circ \gamma$ is a nonconstant geodesic in $\bbH^2$ for each $i = 1, \ldots , r$.

Two unit speed geodesic rays $\gamma$ and $\delta$ in $(\bbH^2)^r$ are said to be asymptotic if there exists a positive number $c$ such that $d(\gamma(t), \delta(t)) \leq c$ for all $t \geq 0$. This is an equivalence relation on the unit speed geodesic rays of $(\bbH^2)^r$. For any unit speed geodesic $\gamma$ of $(\bbH^2)^r$ we denote by $\gamma(+\infty)$ the equivalence class of its positive ray. 

We denote by $\partial ((\bbH^2)^r)$ the set of all equivalence classes of unit speed geodesic rays of $(\bbH^2)^r$. We call $\partial ((\bbH^2)^r)$ the \textit{geometric boundary} of $(\bbH^2)^r$. The \textit{regular boundary} $\partial ((\bbH^2)^r)_{reg}$ of $(\bbH^2)^r$ consists of the equivalence classes of regular geodesics. 

The geometric boundary $\partial ((\bbH^2)^r)$ with the cone topology is homeomorphic to the unit tangent sphere of a point in $(\bbH^2)^r$ (see Eberlein \cite{pE96}, 1.7). (For example $\partial \bbH^2$ is homeomorphic to $S^1$.) The homeomorphism is given by the fact that for each point $x_0$ and each unit speed geodesic ray $\gamma$ in $(\bbH^2)^r$ there exists a unique unit speed geodesic ray $\delta$ with $\delta(0) = x_0$ which is asymptotic to $\gamma$.

The group $\PSL(2,\R)^r$ acts on $(\bbH^2)^r$ by isometries in the following way. For $g = (g_1,\ldots, g_{r}) \in \PSL(2,\R)^r$
$$ g:(\bbH^2)^r \rightarrow (\bbH^2)^r, \quad (z_1, \ldots, z_{r}) \mapsto (g_1 z_1, \ldots, g_{r} z_{r}), $$
where $z_i \mapsto g_i z_i$ is the usual action given by linear fractional transformation, $i = 1, \ldots, {r}$.

The action of $\PSL(2,\R)^r$ can be extended naturally to $\partial ((\bbH^2)^r)$. Let $g$ be in $\PSL(2,\R)^r$ and $\xi$ be a point in the boundary $\partial ((\bbH^2)^r)$. If $\gamma$ is a representative of $\xi$, then $g(\xi)$ is the equivalence class of the geodesic ray $g\circ \gamma$.

We call $g$ \textit{elliptic} if all $g_i$ are elliptic isometries, \textit{parabolic} if all $g_i$ are parabolic isometries and \textit{hyperbolic} if all $g_i$ are hyperbolic isometries. In all the other cases we call $g$ \textit{mixed}. 

If at least one $\ell(g_i)$ is different from zero, then we define the \textit{translation direction} of $g$ as $L(g):= (\ell(g_1): \ldots: \ell(g_{r})) \in \RP^{r-1}$.

\subsection{Decomposition of the geometric boundary of $(\bbH^2)^r$}
In this section we show a natural decomposition of the geometric boundary of $(\bbH^2)^r$ and in particular of its regular part. This is a special case of a general construction for a large class of symmetric spaces (see e.g. Leuzinger~\cite{eL92} and Link~\cite{gL02}). This decomposition plays a main role in this article.

Let $x=(x_1,\ldots,x_{r})$ be a point in $(\bbH^2)^r$. We consider the \textit{Weyl chambers} with vertex $x$ in $(\bbH^2)^r$ given by the product of the images of the geodesics $\delta_i:[0,\infty)\rightarrow \bbH^2$ with $\delta_i(0)=x_i$ for $i = 1,\ldots,r$. The isotropy group in $\PSL(2,\R)^r$ of $x$ is $\mathrm{PSO}(2)^r$. It acts simply transitively on the Weyl chambers with vertex $x$. 

Let $W$ be a Weyl chamber with vertex $x$. In $W$, two unit speed geodesics $\gamma(t) = (\gamma_1(t),\ldots,\gamma_{r}(t))$ and $\tilde{\gamma} = (\tilde{\gamma}_1(t),\ldots,\tilde{\gamma}_{r}(t))$ are different if and only if the corresponding projective points $$\left(d_H(\gamma_1(0),\gamma_1(1)):\ldots:d_H(\gamma_{r}(0),\gamma_{r}(1))\right) \text{ and}$$ $$(d_H(\tilde{\gamma}_1(0),\tilde{\gamma}_1(1)):\ldots:d_H(\tilde{\gamma}_{r}(0),\tilde{\gamma}_{r}(1)))$$ 
are different. Here $d_H$ denotes the hyperbolic distance in $\bbH^2$. The point in $\RP^{r-1}$ given by $\left(d_H(\gamma_1(0),\gamma_1(1)):\ldots:d_H(\gamma_{r}(0),\gamma_r(1))\right)$ is a direction in the Weyl chamber and it is the same as $(\left\|v_1\right\|:\ldots:\left\|v_{r}\right\|)$, where $v = (v_1,\ldots,v_{r}):= \gamma'(0)$ is the unit tangent vector of $\gamma$ in $0$. 

In other words we can extend the action of $Iso_x$ to the tangent space at $x$ in $(\bbH^2)^r$ in such a way that $Iso_x$ maps a unit tangent vector at $x$ onto a unit tangent vector at $x$.

Let $v$ be a unit tangent vector at $x$ in $(\bbH^2)^r$. We denote by $v_i$ the $i$-th projection of $v$ on the tangent spaces at $x_i$, $i=1,\ldots,r$. Then all the vectors $w$ in the orbit of $v$ under $Iso_x$ have $\left\|w_i\right\|=\left\|v_i\right\|$. Even more, the orbit of $v$ under the group $\mathrm{PSO}(2)^r$ consists of all unit tangent vectors $w$ at $x$ such that $\left\|w_i\right\|=\left\|v_i\right\|$ for $i=1,\ldots,r$. Therefore if $v$ is tangent to a regular geodesic, then the orbit of $v$ is homeomorphic to $(S^1)^r \cong \left( \partial \bbH^2 \right)^r $. 

The \textit{regular boundary} $\partial ((\bbH^2)^r)_{reg}$ of $(\bbH^2)^r$ consists of the equivalence classes of regular geodesics. Hence it is identified with $\left( \partial \bbH^2 \right)^r \times \RP^{r-1}_+$ where 
$$ \RP^{r-1}_+ := \left\{(x_1:\ldots:x_{r}) \in \RP^{r-1} \mid x_1 > 0, \ldots, x_{r} > 0 \right\}. $$
Here $x_1,..,x_{r}$ can be thought as the norms of the projections of the regular unit tangent vectors on the simple factors of $(\bbH^2)^r$.

$\left( \partial \bbH^2 \right)^r$ is called the \textit{Furstenberg boundary} of $(\bbH^2)^r$. 

We note that the decomposition of the boundary into orbits under the group $Iso_x$ is independent of the point $x$.

\subsection{The limit set of a group}
Let $x$ be a point and $\{x_n\}_{n\in \N}$ a sequence of points in $(\bbH^2)^r$. We say that $\{x_n\}_{n\in \N}$ converges to a point $\xi \in \partial\left((\bbH^2)^r\right)$ if $\{x_n\}_{n\in \N}$ is discrete in $(\bbH^2)^r$ and the sequence of geodesic rays starting at $x$ and going through $x_n$ converges towards $\xi$ in the cone topology. With this topology, $(\bbH^2)^r \cup \partial\left((\bbH^2)^r\right)$ is a compactification of $(\bbH^2)^r$.

Let $\Gamma$ be a subgroup of $\PSL(2,\R)^r$. We denote by $\Gamma(x)$ the orbit of $x$ under $\Gamma$ and by $\overline{\Gamma(x)}$ - its closure. The \textit{limit set} of $\Gamma$ is $\mathcal{L}_\Gamma:= \overline{\Gamma(x)}\cap \partial\left((\bbH^2)^r\right)$. The limit set is independent of the choice of the point $x$ in $(\bbH^2)^r$. The \textit{regular limit set} is $\mathcal{L}_\Gamma^{reg}:=\mathcal{L}_\Gamma \cap  \partial\left((\bbH^2)^r\right)_{reg}$ and the \textit{singular limit set} is $\mathcal{L}_\Gamma^{sing}:=\mathcal{L}_\Gamma \backslash \mathcal{L}_\Gamma^{reg}$. 

We denote by $F_\Gamma$ the projection of $\mathcal{L}_\Gamma^{reg}$ on the Furstenberg boundary $ \left( \partial \bbH^2 \right)^r$ and by $P_\Gamma$ the projection of $\mathcal{L}_\Gamma^{reg}$ on $\RP^{r-1}_+$. The projection $F_\Gamma$ is the \textit{Furstenberg limit set} of $\Gamma$ and $P_\Gamma$ is the \textit{projective limit set} of $\Gamma$.
\medskip

Let $h\in \Gamma$ be a hyperbolic element or a mixed one with only hyperbolic or elliptic components. There is a unique unit speed geodesic $\gamma$ in $(\bbH^2)^r$ such that $h\circ\gamma(t) = \gamma(t + T_h)$ for a fixed $T_h \in \R_{>0}$ and all $t \in \R$. For $y \in \gamma$, the sequence $h^n(y)$ converges to $\gamma(+\infty)$. Hence also for every $x\in (\bbH^2)^r$, the sequence $h^n(x)$ converges to $\gamma(+\infty)$. Thus $\gamma(+\infty)$ is in $\mathcal{L}_\Gamma$. The sequence $h^{-n}(x)$ converges to $\gamma(-\infty):=-\gamma(+\infty)$ and therefore $\gamma(-\infty)$ is also in $\mathcal{L}_\Gamma$. The points $\gamma(+\infty)$ and $\gamma(-\infty)$ are the only fixed points of $h$ in $\Lim_\Gamma$. The point $\gamma(+\infty)$ is the \textit{attractive} fixed point of $h$ and the point $\gamma(-\infty)$ - the \textit{repulsive} fixed point of $h$. 

If $h$ is hyperbolic, then for all $i=1,\ldots,r$, the projection $p_i \circ \gamma$ is not a point. Hence $\gamma$ is regular and $\gamma(+\infty) \in \mathcal{L}_\Gamma^{reg}$. The point $\gamma(+\infty)$ can be written as $(\xi_F,\xi_P)$ in our description of the regular geometric boundary where 
$$\xi_F := (p_1\circ \gamma(+\infty),\ldots,p_{r}\circ \gamma(+\infty))$$ 
is in the Furstenberg boundary and 
$$\xi_P := (d_H(p_1\circ \gamma(0),p_1\circ \gamma(1)) : \ldots :d_H(p_{r}\circ \gamma(0), p_{r}\circ \gamma(1)))$$ 
is in the projective limit set. Here we note that $\xi_P$ is also equal to 
$$(d_H(p_1\circ \gamma(0),p_1\circ \gamma(T_h)) : \ldots :d_H(p_{r}\circ \gamma(0), p_{r}\circ \gamma(T_h))),$$ 
which is exactly the translation direction of $h$.

Thus the translation direction of each hyperbolic isometry $h$ in $\Gamma$ determines a point in the projective limit set $P_\Gamma$. This point does not change after conjugation with $h$ or after taking a power $h^m$ of $h$, because in these cases the translation direction remains unchanged. 

\medskip

Recall that following Maclachlan and Reid \cite{cM03}, we call a subgroup $\Gamma$ of $\PSL(2,\R)$ \textit{elementary} if there exists a finite $\Gamma$-orbit in $\overline{\bbH^2}:=\bbH^2 \cup \partial \bbH^2$ and \textit{nonelementary} if it is not elementary. Since $\bbH^2$ and $\partial \bbH^2$ are $\Gamma$-invariant, any $\Gamma$-orbit of a point in $\overline{\bbH^2}$ is either completely in $\bbH^2$ or completely in $\partial \bbH^2$. 

We call a subgroup $\Gamma$ of $\PSL(2,\R)^r$ \textit{nonelementary} if for all $i = 1, \ldots, r$,  $p_i(\Gamma)$ is nonelementary, and if for all $g \in \Gamma$ that are mixed, the projections $p_i\circ g$ are either hyperbolic or elliptic of infinite order. Since for all $i = 1, \ldots, r$,  $p_i(\Gamma)$ is nonelementary, $\Gamma$ does not contain only elliptic isometries and thus $\Lim_\Gamma$ is not empty.

This definition of nonelementary is more restrictive than the one given by Link in \cite{gL02}. By Lemma~1.2 in \cite{sG10a} if a subgroup $\Gamma$ of $\PSL(2,\R)^r$ is nonelementary (according to our definition), then it is nonelementary in the sense of Link's definition in \cite{gL02}.

The next theorem is a special case of Theorem 3 from the introduction of \cite{gL02}. It describes the structure of the regular limit set of nonelementary discrete subgroups of $\PSL(2,\R)^r$.

\begin{theorem}[\cite{gL02}]
\label{T:LinkFP}
Let $\Gamma$ be a nonelementary discrete subgroup of the group $\PSL(2,\R)^r$ acting on $(\bbH^2)^r$. If $\mathcal{L}_\Gamma^{reg}$ is not empty, then $F_\Gamma$ is a minimal closed $\Gamma$-invariant subset of $(\partial\bbH^2)^r$, the regular limit set equals the product $F_\Gamma \times P_\Gamma$ and $P_\Gamma$ is equal to the closure in $\RP^{r-1}_+$ of the set of translation directions of the hyperbolic isometries in $\Gamma$. 
\end{theorem}

\subsection{Irreducible arithmetic groups in $\PSL(2,\R)^r$}
\label{S:DefArithm}
In this section, following Schmutz and Wolfart \cite{pS00} and Borel \cite{aB81}, we describe the irreducible arithmetic subgroups of $\PSL(2,\R)^r$.

Let $K$ be a totally real algebraic number field of degree $n = [K:\Q]$ and let $\phi_i$, $i =1, \ldots, n$, be the $n$ distinct embeddings of $K$ into $\R$, where $\phi_1 = id$.

Let $A = \left(\frac{a,b}{K}\right)$ be a quaternion algebra over $K$ such that for $1 \leq i \leq r$, the quaternion algebra $\left(\frac{\phi_i(a),\phi_i(b)}{\R}\right)$ is \textit{unramified}, i.e. isomorphic to the matrix algebra $M(2, \R)$, and for $r < i \leq n$, it is \textit{ramified}, i.e. isomorphic to the Hamilton quaternion algebra $\Ha$. In other words, the embeddings 
\[\phi_i: K \longrightarrow \R, \quad i = 1,\ldots, r\] 
extend to embeddings of $A$ into $M(2,\R)$ and the embeddings 
\[\phi_i: K \longrightarrow \R, \quad i = r+1,\ldots, n\] 
extend to embeddings of $A$ into $\Ha$. Note that the embeddings $\phi_i$, $i = 1,\ldots, r$, of $A$ into the matrix algebra  $M(2,\R)$ are not canonical. 

Let $\mathcal{O}$ be an order in $A$ and $\mathcal{O}^1$ the group of units in $\mathcal{O}$. Define $\Gamma(A,\mathcal{O}):= \phi_1(\mathcal{O}^1) \subset \SL(2,\R)$. The canonical image of $\Gamma(A,\mathcal{O})$ in $\PSL(2,\R)$ is called a group \textit{derived from a quaternion algebra}. The group $\Gamma(A,\mathcal{O})$ acts by isometries on $(\bbH^2)^r$ as follows. An element $g = \phi_1(\varepsilon)$ of $\Gamma(A,\mathcal{O})$ acts via
\[g: (z_1, \ldots, z_{r}) \mapsto (\phi_1(\varepsilon)z_1, \ldots, \phi_{r}(\varepsilon)z_{r}), \]
where $z_i \mapsto \phi_i(\varepsilon)z_i$ is the usual action by linear fractional transformation, $i = 1, \ldots, {r}$.

For a subgroup $S$ of $\Gamma(A,\mathcal{O})$ we denote by $S^*$ the group
$$\{g^*:= (\phi_1(\varepsilon), \ldots, \phi_{r}(\varepsilon))\mid \phi_1(\varepsilon) = g \in S\}.$$
Instead of $(\phi_1(\varepsilon), \ldots, \phi_{r}(\varepsilon))$, we will usually write $(\phi_1(g), \ldots, \phi_{r}(g))$ or, since $\phi_1$ is the identity, even $(g,\phi_2(g), \ldots, \phi_{r}(g))$. The elements $\phi_1(g), \ldots, \phi_{r}(g)$ are called \textit{$\phi$-conjugates}. 

Note that $g^*$ and $S^*$ depend on the chosen embeddings $\phi_i$ of $A$ into $M(2,\R)$. On the other hand, the type of $g^*$ is determined uniquely by the type of $g$. This is given by the following lemma.

\begin{lemma}[\cite{sG10a}]
Let $S$ be a subgroup of $\Gamma(A,\mathcal{O})$ and $S^*$ be defined as above. For an element $g\in S$ the following assertions are true.

1. If $g$ is the identity, then $g^*$ is the identity.

2. If $g$ is parabolic, then $g^*$ is parabolic.

3. If $g$ is elliptic of finite order, then $g^*$ is elliptic of the same order.

4. If $g$ is hyperbolic, then $g^*$ is either hyperbolic or mixed such that, for $i = 1,\ldots,r$, $\phi_i(g)$ is either hyperbolic or elliptic of infinite order.

5. If $g$ is elliptic of infinite order, then its $\phi$-conjugates are hyperbolic or elliptic of infinite order.
\end{lemma}

Hence the mixed isometries in this setting have components that are only hyperbolic or elliptic of infinite order. This justifies the condition in our definition of nonelementary that the projections of all mixed isometries can be only hyperbolic or elliptic of infinite order.

By Borel \cite{aB81}, Section 3.3, all \textit{irreducible arithmetic subgroups} of the group $\PSL(2,\R)^r$ are commensurable to a $\Gamma(A,\mathcal{O})^*$. They have finite covolume. By Margulis \cite{gM91}, for $r\geq2$, all irreducible discrete subgroups of $\PSL(2,\R)^r$ of finite covolume are arithmetic, which shows the importance of the above construction.


\medskip

P. Schmutz and J. Wolfart define in \cite{pS00} arithmetic groups acting on $(\bbH^2)^r$. An \textit{arithmetic group acting on} $(\bbH^2)^r$ is a group $G$ that is commensurable to a $\Gamma(A,\mathcal{O})$. It is finitely generated because it is commensurable to the finitely generated group $\Gamma(A,\mathcal{O})$. Then by Corollary 3.3.5 from \cite{cM03}, the quaternion algebra $AG^{(2)}=  \{\sum a_i g_i \mid a_i \in \Q(\Tr(G^{(2)})), g_i \in G^{(2)}\}$ of the group $G^{(2)}$ generated by the set $\{g^2 \mid g \in G\}$ is an invariant of the commensurability class of $G$. Hence $AG^{(2)}$ is isomorphic to $A$. By Exercise 3.2, No. 1, in \cite{cM03}, $\mathcal{O}G^{(2)}= \{\sum a_i g_i \mid a_i \in \mathcal{O}_{\Q(\Tr(G^{(2)}))}, g_i \in G^{(2)}\}$ is an order in $AG^{(2)}$. Therefore $G^{(2)}$ is a subgroup of $\Gamma(AG^{(2)},\mathcal{O}G^{(2)})$ and we can define $(G^{(2)})^*$. This explains ``acting on $(\bbH^2)^r$" in the name.


The interest of this approach is that it allows us to consider a subgroup $S$ of $G\subset \PSL(2,\R)$ and then the corresponding subgroups ${S^{(2)}}^*$ and ${G^{(2)}}^*$ of $\PSL(2,\R)^r$. If $S$ is a Fuchsian group of finite covolume, then it is called \textit{semi-arithmetic Fuchsian group}.


\subsection{Properties of nonelementary subgroups of $\PSL(2,\R)^r$}

In this section we cite several results that are used later in the article and whose proofs can be found in \cite{sG10a}.

\begin{lemma}[\cite{sG10a}]
\label{L:Schottky}
Let $\Gamma$ be a nonelementary subgroup of $\PSL(2,\R)^r$. Further let $g$ and $h$ be two hyperbolic isometries in $\Gamma$. Then there are hyperbolic isometries $g'$ and $h'$ in $\Gamma$ with $L(g) = L(g')$ and $L(h) = L(h')$ such that the groups generated by the corresponding components are all Schottky groups (with only hyperbolic isometries).
\end{lemma}

\begin{lemma}[\cite{sG10a}]
\label{L:Nonempty}
Let $\Gamma$ be a subgroup of $\PSL(2,\R)^r$ such that all mixed isometries in $\Gamma$ have only elliptic and hyperbolic components and $p_j(\Gamma)$ is nonelementary for one $j\in\{1,\ldots,r\}$. Then $\Gamma$ is nonelementary and $\mathcal{L}_{\Gamma}^{reg}$ is not empty.
\end{lemma}


We define the limit cone of $\Gamma$ to be the closure in $\RP^{r-1}$ of the set of the translation directions of the hyperbolic and mixed isometries in $\Gamma$.

\begin{theorem}[\cite{sG10a}]
\label{T:ConvexAlg}
Let $\Gamma$ be a nonelementary subgroup of an irreducible arithmetic group in $\PSL(2,\R)^r$ with $r \geq 2$. Then $P_\Gamma$ is convex and the closure of $P_\Gamma$ in $\RP^{r-1}$ is equal to the limit cone of $\Gamma$ and in particular the limit cone of $\Gamma$ is convex.
\end{theorem}

\section{Zariski dense subgroups of $\PSL(2,\R)^r$}
\setcounter{theorem}{0}
In this section we give a number-theoretical characterization of Zariski dense nonelementary subgroups of an arithmetic group in $\PSL(2,\R)^r$ (Corollary~\ref{C:CritZariski}). In order to do this we need a classification of all nonelementary Zariski closed subgroups of $\PSL(2,\R)^r$.
\begin{proposition}
\label{P:ZariskiClosure}
Let $\Gamma$ be a nonelementary subgroup of $\PSL(2,\R)^r$. Then the following holds for the Zariski closure of $\Gamma$:
$$ \overline{\Gamma}^Z=\prod_{i=1}^n \Diag_{k_i}(\PSL(2,\R)),$$
where $\Diag_{k_i}(\PSL(2,\R))$ is a conjugate of the diagonal embedding of $\PSL(2,\R)$ in $\PSL(2,\R)^{k_i}$ and where $k_1+\cdots+k_n=r$.
\end{proposition}
\begin{proof}
Let $G$ be the Zariski closure of $\Gamma$ and $p_i$, $i=1,\ldots,r$, the projection of $G$ into the $i$-th factor. 

First we show that $p_i(G)=\PSL(2,\R)$ for all $i=1,\ldots,r$. Let $G_0$ be the connected component of the identity. Note that the group $p_i(G_0)$ is connected and nontrivial. As in the proof of the criterion in \cite{fD00} we remark that $p_i(G_0)$ is normalized  by the nonelementary (and hence Zariski dense) $p_i(\Gamma)$. Hence the projection of the Lie algebra of $G_0$ is normalized by the adjoint action of a Zariski dense group and therefore it is the whole $\mathfrak{sl}(2,\R)$. Thus $p_i$ is open and therefore surjective.


We will prove that every Zariski closed subgroup $G$ of $\PSL(2,\R)^r$ such that its projections to the different simple factors are $\PSL(2,\R)$ has the form $\prod_{i=1}^n \Diag_{k_i}(\PSL(2,\R))$ of the statement of the proposition.

Obviously this statement is true for $r=1$. Let us assume that it is true for $r=1,\ldots,m-1$. We prove it for $r=m$.

We consider the projection $p_1$ on the first factor. The kernel $H:=ker(p_1)$ of $p_1$ is a Zariski closed subgroup of $\PSL(2,\R)^{m}$.

\textit{Claim.} The group $p_i(H)$ is either trivial or the whole group $\PSL(2,\R)$.

Let us assume that there is an element $h=(h_1,\ldots,h_{m})$ in $H$ such that $h_i\neq id$. We choose $g=(g_1,\ldots,g_{m}) \in G$ such that $g_i$ and $h_i$ do not have common fixed points. For some powers $l_g,l_h$, the elements $h_i^{l_h}$ and $g_i^{-l_g}h_i^{l_h}g_i^{l_g}$ generate a Schottky group. This Schottky group is even a subgroup of $H$ because $H$ is a normal subgroup of $G$. Hence there is a Zariski dense subgroup of $\PSL(2,\R)$ which normalizes $p_i(H)$ and thus $p_i(H) = \PSL(2,\R)$.

Without loss of generality, let $p_i(H)=\{id\}$ for $i=2,\ldots,k$ and $\PSL(2,\R)$ for $i=k+1,\ldots,m$. Hence $H$ is canonically isomorphic to a Zariski closed subgroup of $\PSL(2,\R)^{m-k}$ and by induction it is equal to the group $S:=\prod_{i=2}^{n} \Diag_{k_i}(\PSL(2,\R))$ where $k_2+\cdots+k_n=m-k$.

The next step is to show that $G = D \times S$ where $D$ denotes the projection of $G$ on the first $k$ factors. For simplicity of the notation we give the proof just in the case $n=2$ and $k_2=2$. The general case is analogous.

Let us assume the converse, namely that there is an element $(d,s,t)\in G$ such that $d \in D$ and $(s,t) \notin S$, i.e $(s,t)$ is different from the element $(s,s_2)\in S$. Hence $a:=(d,id,u:=ts_2^{-1})$ is in $G$ and $u\neq id$. We choose $v\in \PSL(2,\R)$ that does not commute with $u$ and consider the corresponding element $b:=(id, v_1, v)$ in $H$. The commutator $[a,b] = (id, id, uvu^{-1}v^{-1})$ is an element in $H$ that contradicts its structure.

The final step is to show that $D$ is a conjugate of the diagonal embedding of $\PSL(2,\R)$ in $\PSL(2,\R)^k$. Since $ker(p_1|_D)=\{id\}$, $\PSL(2,\R)$ is isomorphic to $D$. And since for each $i=1,\ldots, k$ the projection $p_i$ is surjective, we have that $D= \{(x, \phi_2(x),\ldots,\phi_k(x)) \}$ where $\phi_i$ are surjective homomorphisms. Even more, $\phi_i$ are isomorphisms because if the kernel is not trivial it can only be $\PSL(2,\R)$, which is is obviously not possible.

The projections $p_i|_D$ are continuous isomorphisms and open maps as remarked in the beginning of the proof. Hence $\phi_i$ are continuous isomorphisms and hence conjugations of $\PSL(2,\R)$ by elements in $\GL(2,\R)$.
\end{proof}

For a group $S$ we denote as usual by $S^{(2)}$ its subgroup generated by the set $\{g^2 \mid g \in S\}$. If $S$ is a finitely generated nonelementary subgroup of $\PSL(2,\R)$ then $S^{(2)}$ is a finite index normal subgroup of $S$.

\begin{corollary}
\label{C:CritZariski}
Let $\Gamma$ be a nonelementary subgroup of an irreducible arithmetic group $\Delta$ in $\PSL(2,\R)^r$. Then $\Gamma$ is Zariski dense in $\PSL(2,\R)^r$ if and only if the fields generated by $\Tr(p_i(\Gamma^{(2)}))$ and $\Tr(p_i(\Delta^{(2)}))$ are equal for one and hence for all $i \in \{1,\ldots,r\}$.
\end{corollary}
\begin{proof}
Since $\Delta$ is arithmetic, it is commensurable with an arithmetic group derived from a quaternion algebra $\Gamma(A,\mathcal{O})^*$. Hence there is $k\in\N$ such that, for each $g = (g_1,\ldots,g_{q+r})$ in $\Gamma$, $g^k$ is in $\Gamma(A,\mathcal{O})^*$.

There is a subgroup $S$ of $\Gamma(A,\mathcal{O})$ such that $S^* = \Gamma \cap \Gamma(A,\mathcal{O})^*$. The group $\Gamma$ is commensurable with its subgroup $S^*$. Hence $p_1(\Gamma)$ and $S$ are also commensurable. The group $S^*$ is finitely generated because it is a finite index subgroup of the finitely generated group $\Gamma$. (This follows from the Schreier Index Formula, see for example the book of Stillwell \cite{jS93}, 2.2.5.) The group $S^*$ is also nonelementary because $\Gamma$ is nonelementary: Let $g$ and $h$ be two loxodromic isometries that generate a Schottky group in $\Gamma$. The isometries $g^{k}$ and $h^{k}$ are in $S^*$. Then $g^{k}$ and $h^{k}$ generate a Schottky subgroup of $S^*$.

Thus $S$ is a finitely generated nonelementary subgroup of $\Gamma(A,\mathcal{O})$.

We prove now the statement for $S^*$ and $\Gamma(A,\mathcal{O})^*$. We prove it for the first projection, i.e. we show that $S^*$ is Zariski dense if and only if the trace field of $S^{(2)}$ is the same as the trace field of $\Gamma(A,\mathcal{O})^{(2)}$. The general case is analogous.

First we assume that $(S^{(2)})^*$ is not Zariski dense. By Proposition~\ref{P:ZariskiClosure} there is in its Zariski closure at least one block $\Diag_{k}(\PSL(2,\R))$ with $k \geq 2$. Hence there are $\phi_i$ and $\phi_j$ with $i\neq j$ such that $\phi_i(\Tr(S^{(2)}))=\phi_j(\Tr(S^{(2)}))$. Since $\phi_i$ and $\phi_j$ are different for the trace field of $\Gamma(A,\mathcal{O})^{(2)}$, it follows that the trace field of $S^{(2)}$ is not the same as the trace field of $\Gamma(A,\mathcal{O})^{(2)}$.

In order to prove the converse we assume that the trace field $F$ of $S^{(2)}$ is not the same as the trace field $K$ of $\Gamma(A,\mathcal{O})^{(2)}$, i.e. it is a nontrivial subfield. Hence there is a $\phi_i\neq id$ such that $\phi_i|_F = id|_F$. Since $\phi_i(S^{(2)})$ is nonelementary, the group $\phi_i(\Gamma(A,\mathcal{O})^{(2)})$ is also nonelementary and hence $i \leq r$. Therefore the groups $S^{(2)}$ and $\phi_i(S^{(2)})$ are conjugated by a matrix in $\GL(2,\R)$ (see for example Sampson \cite{jS76}). This implies that $(S^{(2)})^*$ is not Zariski dense.
\end{proof}

\section{The limit set of modular embeddings of Fuchsian groups}
\label{S:Modular}
\setcounter{theorem}{0}
As the following theorem shows the projection of the regular limit set into the Furstenberg boundary of the embedding in a lattice in $\PSL(2,\R)^r$ of an arithmetic Fuchsian group is homeomorphic to a circle. This is an exemple of a ``small" group with a ``small" limit set.

\begin{theorem}[\cite{sG10a}]
\label{T:MainCharactAlgR}
Let $\Delta$ be an irreducible arithmetic subgroup of $\PSL(2,\R)^r$ with $r\geq2$ and $\Gamma$ a finitely generated nonelementary subgroup of $\Delta$. Then $\mathcal{L}_{\Gamma}$ is homeomorphically embedded in a circle if and only if $p_j(\Gamma)$ is contained in an arithmetic Fuchsian group for some $j\in\{1,\ldots,r\}$. 
\end{theorem}

In this section we provide examples of ``small" groups with a ``big" limit set, i.e. groups for which the projection of the regular limit set into the Furstenberg boundary is the whole Furstenberg boundary and even more for which the limit set is of nonempty interior.

Let $S$ be a semi-arithmetic subgroup of $\PSL(2,\R)$ that is a subgroup of some $\Gamma(A,\mathcal{O})$ as defined in \S \ref{S:DefArithm}. We set $\Gamma:=S^*$ a subgroup of $\PSL(2,\R)^r$. Then $S$ is said to have a \textit{modular embedding} if  for the natural embedding $f:S \rightarrow S^*=\Gamma$ there exists a holomorphic embedding $F:\bbH^2 \rightarrow (\bbH^2)^r$ with
$$ F(T z) =  f(T)F(z), \quad \text{for all} \quad T\in S \quad \text{and all} \quad z \in \bbH^2.$$
Examples of strictly semi-arithmetic groups admitting modular embeddings are Fuchsian triangle groups (see Cohen and Wolfart \cite{pC90}). 

According to Theorem 3 by Schmutz and Wolfart in \cite{pS00}, if $\phi_i$ is not a conjugation then $\phi_i(S)$ is not a Fuchsian group. Since every nondiscrete subgroup of $\PSL(2,\R)$ has elliptic elements of infinite order (see Theorem~8.4.1 in Beardon \cite{aB95}), the nondiscrete group $\phi_i(S)$ contains elliptic elements of infinite order. (Here the $\phi_i$ are the same as the ones from the definition of $S^*$.) 

Additionally, Corollary 5 by Schmutz and Wolfart \cite{pS00} states that for all $g=(g_1,\ldots,g_r)$ in $\Gamma$ such that $g_1$ is hyperbolic, we have the inequality $\tr(g_i)<\tr(g_1)$ for $i=2,\ldots,r$. Hence the Galois conjugates of the traces of hyperbolic elements are different, i.e. $|\phi_i(\tr(g_1))|\neq|\phi_j(\tr(g_1))|$ if $i\neq j$. Hence no trace in $\Tr(S)$ of a hyperbolic element is contained in a proper subfield of the trace field $\Q(\Tr(S))$ of $S$. In particular there are no hyperbolic elements in $S$ with integer traces.

\begin{lemma}
\label{L:MinPower}
Let $\Gamma$ be a subgroup of an irreducible arithmetic subgroup of $\PSL(2,\R)^r$ with $r\geq 2$ such that $p_j(\Gamma)$ is a semi-arithmetic Fuchsian group admitting a modular embedding and $r$ is the smallest power for which $p_j(\Gamma)$ has a modular embedding in an irreducible arithmetic subgroup of $\PSL(2,\R)^r$. Then for all the other $i = 1,\ldots, r$, $i \neq j$, the group $p_i(\Gamma)$ is not discrete and for every $g_j \in p_j(\Gamma)$ hyperbolic $\Q(\tr(g_j^2))=\Q(\Tr(p_j(\Gamma^{(2)})))$.
\end{lemma}


\subsection{The Furstenberg limit set}

The main result of this section is Theorem~\ref{T:BigFurstenberg}.It proves in particular for the groups satisfying the conditions of Lemma~\ref{L:MinPower} that their Furstenberg limit set is the whole Furstenberg boundary. First we need to prove the following technical lemma.

\begin{lemma}
\label{L:OneElliptic}
Let $\Gamma$ be a nonelementary subgroup of an arithmetic group in $\PSL(2,\R)^r$ with $r\geq 2$ and let $g=(g_1,\ldots,g_r)\in\Gamma$ be a mixed isometry with $g_1,\ldots, g_{k-1}$ hyperbolic and $g_k,\ldots, g_r$ elliptic of infinite order where $2 \leq k \leq r-1$ and such that $\Q(\tr(g_1))=\Q(\tr(g_1^n))$ for every integer $n\neq 0$. If $\tr(g_k)\neq \tr(g_{k+1})$, then there exists a mixed isometry $\tilde{g}\in\Gamma$ such that $\tilde{g}_1,\ldots, \tilde{g}_k$ are hyperbolic and $\tilde{g}_{k+1}$ is elliptic of infinite order.
\end{lemma}
\begin{proof}

Since $\Gamma$ is nonelementary, by Lemma~\ref{L:Nonempty} and by the fact that the fixed points of hyperbolic isometries are dense in $\Lim_\Gamma$, there is $h\in \Gamma$ that is a hyperbolic isometry.
By Lemma~\ref{L:Schottky}, we can assume without loss of generality that for all $i=1,\ldots,k-1$, $h_i$ and $g_i$ generate a Schottky group.

The idea is to find $m\in\N$ such that $g_k^mh_k$ is hyperbolic and $g_{k+1}^mh_{k+1}$ is elliptic (of infinite order).

By \S 7.34 in the book of Beardon~\cite{aB95}, the isometries $h_i$, $i=k,k+1$, can be represented as $h_i=\sigma_{i,2}\sigma_{i,1}$ where $\sigma_{i,j}$ are the reflections in geodesics $L_{i,j}$ that are orthogonal to the axis of $h_i$ and the distance between them equals $\ell(h_i)/2$. By \S 7.33 in \cite{aB95}, the elliptic isometries $g_i^m$ can be represented as $g_i^m=\sigma_{i,4}\sigma_{i,3}$ where $\sigma_{i,j}$ are the reflections in $L_{i,j}$ that pass through the fixed point of $g_i^m$ and the angle between them equals the half of the angle of rotation of $g_i$. We can choose $L_{i,2}$ and $L_{i,3}$ to be the same geodesic by choosing them to be the unique geodesic passing through the fixed point of $g_i^m$ that is orthogonal to the axis of $h_i$. Then $\sigma_{i,2}=\sigma_{i,3}$ and
$$ g_i^mh_i = (\sigma_{i,4}\sigma_{i,3})(\sigma_{i,2}\sigma_{i,1})=\sigma_{i,4}\sigma_{i,1}.$$
Hence if $L_{i,1}$ and $L_{i,4}$ do not intersect and do not have a common point at infinity, the isometry $g_i^mh_i$ is hyperbolic, and if they intersect, $g_i^mh_i$ is elliptic. By the next claim we can choose $m$ so that $g_k^m$ is a rotation of a very small angle such that $L_{k,4}$ does not intersect $L_{k,1}$ and $g_{k+1}^m$ is a rotation such that $L_{k+1,4}$ intersects $L_{k+1,1}$.

\textbf{Claim.} Every orbit of $(g_k,g_{k+1})$ on the 2-torus $(\partial\bbH^2)^2$ is dense.

Since $\tr(g_k)\neq \tr(g_{k+1})$ and $g_k$ and $g_{k+1}$ are elliptic of infinite order, they are rotations of angles $2\alpha\pi$ and $2\beta\pi$ respectively, i.e. $\tr(g_k)=|2\cos(\alpha \pi)|$ and $\tr(g_{k+1})=|2\cos(\beta \pi)|$, where $\alpha$ and $\beta$ are different irrational numbers in the interval $[0,\pi)$ and $\cos(\alpha\pi)$ is a Galois conjugate of $\cos(\beta\pi)=:\phi(\cos(\alpha\pi))$. The orbit of a point $(\xi_k,\xi_{k+1})$ in $(\partial\bbH^2)^2$ under $(g_k,g_{k+1})$ can be represented as the orbit of a point in the flat torus $S^1\times S^1$ under translations with translation vector $(2\alpha,2\beta)$. By Theorem 442 in the book of Hardy and Wright~\cite{gH62}, the orbit of a point is dense if $2\alpha$, $2\beta$ and $1$ are linearly independent over $\Q$. 

Let us assume that $2\alpha$, $2\beta$ and $1$ are linearly dependent over $\Q$, i.e. $2m\alpha + 2n\beta = q$ for some integers $q,m,n$. 

For a positive integer $l$, we have $\cos(lx)=P_l(\cos x)$ where $P_l$ is a polynomial of degree $l$. Hence 
\begin{eqnarray*}
\phi(\cos(l\alpha\pi)) &=&\phi(P_l(\cos(\alpha\pi)))=P_l(\phi(\cos(\alpha\pi)))\\
 &=&P_l(\cos(\beta\pi))=\cos(l\beta\pi).
\end{eqnarray*}

This means that a conjugate of $\cos(n\alpha\pi)$ is 
\begin{eqnarray*}
\phi(\cos(n\alpha\pi)) &=&\cos(n\beta\pi)=\cos((q-m\alpha)\pi)\\
&=&\cos(q\pi)\cos(m\alpha\pi)+\sin(q\pi)\sin(m\alpha\pi).
\end{eqnarray*}

Up to multiplying by two the equality $2m\alpha + 2n\beta = q$, one can assume that $q$ is even and hence $\sin(q\pi)=0$ and $\cos(q\pi)=1$. Thus $\phi(\cos(n\alpha\pi)) = \cos(m\alpha\pi)$. 

If $m=n$, then $\phi(\cos(n\alpha\pi)) = \cos(n\alpha\pi)$. Hence $\phi_k(g_1^n)=2\cos(n\alpha\pi)=\phi_{k+1}(g_1^n)$, which is impossible because $\phi_k(g_1)=2\cos(\alpha\pi)\neq\phi_{k+1}(g_1)$ and  $\Q(\tr(g_1))=\Q(\tr(g_1^n))$. Therefore $m\neq n$.

We will show that $\phi(\cos(n\alpha\pi)) = \cos(m\alpha\pi)$ is impossible by finding $c$ different conjugates of $\cos(n^c\alpha\pi)$ where $c$ is an arbitrary integer.
\begin{eqnarray*} \phi(\cos(n^c\alpha\pi))& = & \phi(P_{n^{c-1}}(\cos(n\alpha\pi))) = P_{n^{c-1}}(\phi(\cos(n\alpha\pi)))\\
&=& P_{n^{c-1}}( \cos(m\alpha\pi)) = \cos(mn^{c-1}\alpha\pi).
\end{eqnarray*}
By induction we see that for every integer $d$ with $0<d\leq c$,
$$ \phi^d(\cos(n^c\alpha\pi)) =\cos(m^dn^{c-d}\alpha\pi).$$
All $\phi^d(\cos(n^d\alpha\pi))$ are conjugate and different because $\cos(m^dn^{c-d}\alpha\pi)$ are different for different $d$. This contradicts the fact that $\Q(\Tr(\Gamma))$ is a finite extension of $\Q$.

Thus $2\alpha$, $2\beta$ and $1$ are linearly independent over $\Q$. Therefore the orbit of $(\xi_k,\xi_k)$ under $(g_k,g_{k+1})$ is dense in $\partial\bbH^2\times\partial\bbH^2$.
\end{proof}


\begin{theorem}
\label{T:BigFurstenberg}
Let $\Gamma$ be a subgroup of an irreducible arithmetic group $\Delta$ in $\PSL(2,\R)^r$ with $r \geq 2$ such that, for one $j\in \{1,\ldots,r\}$, $p_j(\Gamma)$ is a cofinite Fuchsian group and for all the others $i = 1,\ldots, r$, $i \neq j$, the group $p_i(\Gamma)$ is not discrete and such that,  for every $g_j \in p_j(\Gamma)$ hyperbolic, $\Q(\tr(g_j^2))=\Q(\Tr(p_j(\Gamma^{(2)})))$.

Then $F_\Gamma$ is the whole Furstenberg boundary $(\partial \bbH^2)^r$.
\end{theorem}
\begin{proof} Without loss of generality we can assume that $j=1$.

We denote $\Gamma_i :=p_i(\Gamma)$, for $i=1,\ldots,r$.
 
First we remark that by Lemma~\ref{L:Nonempty} the regular limit set of $\Gamma$ is not empty and hence there is $\xi = (\xi_1,\ldots,\xi_r) \in F_\Gamma$. Without loss of generality, we can assume that $\xi$ is the projection in the Furstenberg boundary of the attractive fixed point of a hyperbolic isometry in $\Gamma$. By Theorem~\ref{T:LinkFP}, $F_\Gamma$ is the minimal closed $\Gamma$-invariant subset of the Furstenberg boundary $(\partial \bbH^2)^r$. Therefore $F_\Gamma = \overline{\Gamma(\xi)}$. The idea is to show that $\overline{\Gamma(\xi)}$ is the whole Furstenberg boundary $(\partial \bbH^2)^r$ and hence $F_\Gamma = (\partial \bbH^2)^r$.

We will mainly use the fact that except for $\Gamma_1$ all the other $\Gamma_i$ are non discrete and in particular that then they contain elliptic elements of infinite order. An elliptic element $e$ of infinite order acts on $\partial\bbH^2$ as a "rotation" of irrational angle. That is why the orbit of a point in $\partial\bbH^2$ under the action of $e$ is dense in $\partial\bbH^2$.

We further note that the restriction of every $\phi_i$, $i=2,\ldots,r$, to $\Tr(\Gamma_1^{(2)})$ is not the identity. Therefore, since  $\Q(\tr(g_1^2))=\Q(\Tr(\Gamma_1^{(2)}))$  for every $g_1 \in \Gamma_1$ hyperbolic, the Galois conjugates of every hyperbolic element $g_1 \in\Gamma_1$ are different.

We will show by induction that $(\partial \bbH^2)^r = \overline{\Gamma(\xi)}$. By abuse of notation, when it is clear from the context, we will note the projection of $\Gamma$ to the first $m$ factors also with $\Gamma$.

\medskip
\textbf{Case $m=2$:} Let $g = (g_1,g_2)\in \Gamma$ be a transformation such that $g_2$ is an elliptic transformation of infinite order. Such $g_2$ exists because $\Gamma_2$ is not discrete. The isometry $g_1$ is hyperbolic because $\Gamma_1$ is discrete. 

Let $\eta_1$ be the attractive fixed point of $g_1$. First we note that $\{\eta_1\}\times\partial\bbH^2$ is in $\overline{\Gamma(\xi)}$. The reason is that for any point $\eta_2 \in \bbH^2$ there is a sequence $\{n_k\}$ of powers such that $g_2^{n_k}(\xi_2)\longrightarrow \eta_2$ when $n_k\longrightarrow \infty$ and additionally, because of the dynamics of the hyperbolic isometries, $g_1^{n_k}(\xi_1)\longrightarrow \eta_1$ when $n_k\longrightarrow \infty$. The only problem could be that $\xi_1$ is the repelling fixed point of $g_1$. In this case we consider $g^{-1}$ instead of $g$. 

Now let $\zeta = (\zeta_1,\zeta_2)$ be a point in $\partial\bbH^2\times\partial\bbH^2$. Since $\Gamma_1$ is a cofinite Fuchsian group, $\Lim_{\Gamma_1}$ is the whole boundary of $\bbH^2$. It is also the minimal closed $\Gamma_1$-invariant subset of $\partial\bbH^2$(see Theorem~5.3.7 in \cite{aB95}). Therefore $\Lim_{\Gamma_1}=\overline{\Gamma_1(\eta_1)}$ and there is a sequence of elements $\{h_n\}$ in $\Gamma$ such that $(h_n)_1(\eta_1)\stackrel{n\rightarrow \infty}\longrightarrow \zeta_1$.

The points $(\eta_1,(h_n)_2^{-1}(\zeta_2))$ are points in $\overline{\Gamma(\xi)}$ because it contains $\{\eta_1\}\times\partial\bbH^2$. Therefore all the points $((h_n)_1(\eta_1), \zeta_2)$ are in $\overline{\Gamma(\xi)}$ and hence their limit $\zeta=\lim_{n\rightarrow \infty}((h_n)_1(\eta_1), \zeta_2)$ is also in $\overline{\Gamma(\xi)}$. Thus $\overline{\Gamma(\xi)}=(\partial \bbH^2)^2$.

\medskip
\textbf{Case $m \geq 3$:} Let us assume that $(\partial \bbH^2)^{m-1} = \overline{\Gamma((\xi_1,\ldots,\xi_{m-1}))}$. The aim is to prove that $(\partial \bbH^2)^{m} = \overline{\Gamma((\xi_1\ldots,\xi_{m}))}$.

First we show that there exists $g=(g_1,\ldots,g_m)\in \Gamma$ such that $g_1,\ldots,g_{m-1}$ are hyperbolic and $g_m$ is elliptic of infinite order. Indeed, since $\Gamma_m$ is not discrete, there is $g'=(g'_1,\ldots,g'_m )\in \Gamma$ such that $g'_m$ is elliptic of infinite order. If all the other $g'_i$ are hyperbolic, then $g=g'$. Otherwise, since $\Q(\tr({g'}_1^2))=\Q(\tr({g'}_1^{2n}))$ for every integer $n\neq 0$ and also all the Galois conjugates of the hyperbolic element ${g'}_1^2$ are different, we can apply Lemma~\ref{L:OneElliptic} to ${g'}_1^2$. In this way we find $\tilde{g'}\in \Gamma$ such that $\tilde{g'}_m$ is elliptic and among the other $\tilde{g'}_i$ there are strictly more hyperbolic elements than in the initial isometry. Thus, after possibly using Lemma~\ref{L:OneElliptic} several times, we can find $g=(g_1,\ldots,g_m)\in \Gamma$ such that $g_1,\ldots,g_{m-1}$ are hyperbolic and $g_m$ is elliptic of infinite order.

Let $\eta_i$ denote the attractive fixed point of $g_i$ for $i=1,\ldots,m$. We can assume that $\xi_i$ is none of the fixed points of $g_i$. The reason is that $\xi$ is the projection on the Furstenberg boundary of the attractive fixed point of an element $h$ in $\Gamma$. By Lemma~\ref{L:Schottky}, we can find $h'$ such that $h'_i$ and $g_i$ do not have any common fixed point for $i=1,\ldots,m-1$. Then instead of the attractive fixed point of $h$ we can take the attractive fixed point of $h'$. 

As in the previous case we see that $\{\eta_1\}\times\ldots\times\{\eta_{n-1}\}\times\partial\bbH^2$ is a subset of $\overline{\Gamma(\xi)}$ and that for any point $(\zeta_1,\ldots,\zeta_m)$ in the closure of the $\Gamma$-orbit of $\{\eta_1\}\times\ldots\times\{\eta_{m-1}\}\times\partial\bbH^2$, the points in $\{\zeta_1\}\times\ldots\times\{\zeta_{m-1}\}\times\partial\bbH^2$ are also in the closure of the orbit. Hence by the induction hypothesis, $(\partial \bbH^2)^{m} = \overline{\Gamma((\xi_1\ldots,\xi_{m}))}$.
\end{proof}
\begin{corollary} Let $\Gamma$ be as in Theorem~\ref{T:BigFurstenberg}. Then $\Lim_\Gamma$ is connected. 
\end{corollary}
\begin{proof}
Since $\Lim_\Gamma$ is the closure of the attractive fixed points of the hyperbolic isometries in $\Gamma$, it is also the closure of $\Lim_\Gamma^{reg}$. Hence if $\Lim_\Gamma^{reg}$ is connected, then $\Lim_\Gamma$ is connected too. $\Lim_\Gamma^{reg} = F_\Gamma \times P_\Gamma$ is connected because $F_\Gamma$ is a torus and $P_\Gamma$ is convex by Theorem~\ref{T:ConvexAlg}.
\end{proof}

\begin{corollary}
Let $\Delta$ be an irreducible arithmetic subgroup of $\PSL(2,\R)^r$ with $r\geq2$ and $\Gamma$ a subgroup of $\Delta$ such that $p_j(\Gamma)$ is a Fuchsian group admitting a modular embedding. Then $F_\Gamma$ is a $q$-dimensional torus, where $q$ is the smallest power for which $p_j(\Gamma)$ has a modular embedding in an irreducible arithmetic subgroup of $\PSL(2,\R)^q$.
\end{corollary}
\begin{proof}
Without loss of generality we can assume that $j=1$. We denote $S:=p_1(\Gamma)$.

Let $l:= \Q(\Tr(S^{(2)}))$ and $\sigma_1,\ldots,\sigma_m$ its Galois isomorphisms. Since $q \leq m$ is the smallest power for which $S^{(2)}$ can be embedded in an irreducible arithmetic subgroup of $\PSL(2,\R)^q$, we have that $\sigma_i(\Tr(S^{(2)}))$ is not bounded for $i=1,\ldots, q$ and bounded for $i=q+1,\ldots, m$. Hence for every $\sigma_a$ with $a \in \{1,\ldots,q\}$ there is a $\phi_b$ with $b \in \{1,\ldots,r\}$ such that $\left.\phi_b\right|_l=\sigma_a$ and also for every $\phi_b$ with $b \in \{1,\ldots,r\}$ the restriction $\left.\phi_b\right|_l$ equals $\sigma_a$ for some $a \in \{1,\ldots,q\}$.

The rest of the proof is analogous to the proof of Lemma 5.16 in \cite{sG10a}, which proves the statement in the case when $p_j(\Gamma)$ is an arithmetic Fuchsian group.
\end{proof}

\subsection{The projective limit set}
Recall that by Corollary 5 by Schmutz and Wolfart \cite{pS00} for all $g=(g_1,\ldots,g_r)$ in $\Gamma$ such that $g_1$ is hyperbolic, we have the inequality $\tr(g_i)<\tr(g_1)$ for $i=2,\ldots,r$. Hence for every hyperbolic (or mixed) $g$ in $\Gamma$ the inequalities $\ell(g_i)<\ell(g_1)$ for $i=2,\ldots,r$ hold. Thus the translation direction $L(g)$ is in $\{(x_1:\ldots:x_r)\mid x_1 > x_i \text{ for all } i= 2,\ldots,r\}$.

Since $\Gamma$ is nonelementary, by Theorem~\ref{T:LinkFP} it follows that the translation directions of the hyperbolic isometries in $\Gamma$ are dense in $P_\Gamma$. So we have the following proposition.

\begin{proposition}
Let $\Gamma < \PSL(2,\R)^r$ be a modular embedding of a semi-arithmetic Fuchsian group $S$. Then 
$$P_\Gamma \subseteq \{(x_1:\ldots:x_r)\mid x_1 \geq x_i \text{ for all } i= 2,\ldots,r\}.$$
\end{proposition}
\begin{rem}(i) In the above proposition, $r$ is not necessarily the smallest power for which $\Gamma$ has a modular embedding in an irreducible arithmetic subgroup of $\PSL(2,\R)^r$

(ii) There are examples of semi-arithmetic groups that do not admit modular embeddings and for which the proposition is not true. E.g. the strictly semi-arithmetic examples from Theorem 1 in \cite{pS00}.
\end{rem}

As this proposition already shows, $P_\Gamma$ is not $\RP^{r-1}_+$ and hence the limit set $\Lim_\Gamma$ is not the whole geometric boundary of $(\bbH^2)^r$. Nevertheless, we can still ask the question whether $P_\Gamma$ contains an open subset of $\RP^{r-1}_+$ or not.


The next theorem gives sufficient conditions for the limit set to be of nonempty interior.

\begin{theorem}
\label{T:BigLimitSet}
Let $\Gamma$ be a subgroup of an irreducible arithmetic group $\Delta$ in $\PSL(2,\R)^r$ with $r \geq 2$ such that, for one $j\in \{1,\ldots,r\}$, $p_j(\Gamma)$ is a cofinite Fuchsian group and for all the others $i = 1,\ldots, r$, $i \neq j$, the group $p_i(\Gamma)$ is not discrete and such that, for every $g_j \in p_j(\Gamma)$ hyperbolic, $\Q(\tr(g_j^2))=\Q(\Tr(p_j(\Gamma^{(2)})))$. Then $\Lim_\Gamma$ contains a subset homeomorphic to $D^{2r-1}$ (the $(2r-1)$-dimensional ball).
\end{theorem}

\begin{proof}
It is enough to show that the regular limit set $\Lim_\Gamma^{reg} = F_\Gamma \times P_\Gamma$ contains a subset homeomorphic to $D^{2r-1}$. By Theorem~\ref{T:BigFurstenberg}, the Furstenberg limit set equals the Furstenberg boundary and hence contains an $r$-dimensional ball. It remains to show that $P_\Gamma$ contains an $(r-1)$-dimensional ball. By Benoist's theorem in Section 1.2 in \cite{yB97}, for Zariski dense groups $\Gamma$, the projective limit set $P_\Gamma$ is of nonempty inerior. Hence it suffices to prove that $\Gamma$ is Zariski dense.

Without loss of generality we assume that $j=1$. For each $i = 2,\ldots, r$, there is a hyperbolic element $h$ in $p_1(\Gamma)$ such that $\phi_i(h)$ is elliptic of infinite order and thus $\tr(h) \neq \phi_i(\tr(h))$. Additionally, for $i=r+1,\ldots, n$, $\phi_i(\Tr(p_1(\Gamma^{(2)}))) \subseteq [-2,2]$.  Therefore, for $i=2,\ldots,n$, $\phi_i\left|_{\Tr(p_1(\Gamma^{(2)}))}\right. \neq id$ and so the fields generated by $\Tr(p_1(\Gamma^{(2)}))$ and $\Tr(p_1(\Delta^{(2)}))$ coinside. Hence by Corollary~\ref{C:CritZariski}, the group $\Gamma$ is Zariski dense.
\end{proof}

\section{Subgroups with parabolic elements}
The existance of parabolic elements in the group gives additional information on the projective limit set (see Proposition~\ref{P:Parabolic}). As a corollary, we propose an alternative proof in this case of Theorem~\ref{T:BigLimitSet} without using the result of Benoist \cite{yB97}.

\begin{proposition}
\label{P:Parabolic}
Let $\Delta$ be an irreducible arithmetic subgroup of $\PSL(2,\R)^r$ with $r\geq2$ and $\Gamma$ a nonelementary subgroup of $\Delta$. If $\Gamma$ contains a parabolic element, then $P_\Gamma$ contains the point $(1:\ldots:1)$.
\end{proposition}
\begin{proof}
Since $\Delta$ is arithmetic, it is commensurable with an arithmetic group derived from a quaternion algebra $\Gamma(A,\mathcal{O})^*$ where $\Gamma(A,\mathcal{O})$ is a subgroup of $\PSL(2,\R)$. Hence there is $k\in\N$ such that, for each $g = (g_1,\ldots,g_{r})$ in $\Gamma$, $g^k$ is in $\Gamma(A,\mathcal{O})^*$.

There is a subgroup $S$ of $\Gamma(A,\mathcal{O})$ such that $S^* = \Gamma \cap \Delta \cap \Gamma(A,\mathcal{O})^*$. The group $\Gamma$ is commensurable with the subgroup $S^*$ and hence $S$ is nonelementary. The group $S^*$ also contains a parabolic isometry.

We will show that $\mathcal{L}_{S^*}^{reg}$ is not empty and $P_{S^*}$ contains $(1:\ldots:1)$.

Let $T_u$ be a hyperbolic isometry in $S$. The connection between the translation length of $T_u$ and its trace is given by $\tr(T_u) = 2 \cosh (\ell(T_u)/2)$. The translation length of $T_u$ is equal to the length of the only simple closed geodesic in $\left\langle T_u\right\rangle\backslash \bbH^2$. For a Fuchsian group $\Gamma$ there is a bijection between its hyperbolic elements and the closed geodesics in $\Gamma\backslash \bbH^2$. Hence for the length of every closed geodesic in $\Gamma\backslash \bbH^2$ there is at least one hyperbolic transformation in $\Gamma$ with this translation length.

A \textit{Y-piece} is a surface of constant curvature -1 and of signature (0, 3), i.e. homeomorphic to a topological sphere with three points removed. By Leuzinger and the autor \cite{sG08}, Corollary 2.3, which is true not only for Fuchsian groups but also for nonelementary (non-discrete) subgroups of $\PSL(2,\R)$, an element $T_v \in S$ exists such that $\left\langle T_u,T_v\right\rangle\backslash \bbH^2$ contains a Y-piece with one cusp and two boundary geodesics of length $\ell(T_u)$ and $\ell(T_v)$. By Schmutz \cite{pS96}, Lemma 1, for all $n \in \N$, this Y-piece contains a closed geodesic such that the corresponding hyperbolic isometry $T_n$ satisfies
$$ \tr(T_n) = n (\tr(T_u) + \tr(T_v)) - \tr(T_u).$$
For $n$ big enough, all $\phi$-conjugates $\phi_i(T_n)$ of $T_n$ are hyperbolic because 
$$\tr(\phi_i(T_n)) = |\phi_i(\tr(T_n))| = |n (\phi_i(\tr(T_u)) + \phi_i(\tr(T_v))) - \phi_i(\tr(T_u))|$$
is unbounded when $n$ is unbounded. (In this way we found a hyperbolic element in $S^*$, i.e. a hyperbolic element in $S$ such that all its $\phi$-conjugates are hyperbolic.)

We set $A:= \tr(T_u) + \tr(T_v)$ and $B := \tr(T_u)$. Then
$$\ell(T_n) = 2 \arcosh \frac{\tr(T_n)}{2} = 2 \arcosh \frac{nA - B}{2} \stackrel{n\rightarrow \infty}\approx 2 \ln(nA - B).$$
Hence for big $n$, the translation direction of $T_n$ is asymptotic to 
$$ (2 \ln(n) + 2 \ln(A):\ldots: 2 \ln(n) + 2 \ln(|\phi_i(A)|)),$$
which converges towards $(1:\ldots:1)$. Thus $(1:\ldots:1)$ is in $P_{S^*}$ and hence $(1:\ldots:1)$ is in $P_\Gamma$.
\end{proof}

\begin{rem}
Examples of groups  with $r=2$ are the triangle groups $(5,\infty,\infty)$ and $(2,5,\infty)$. For them we have also that their limit cone $C_\Gamma$ is a subset of $P:=\{(x_1:x_2)\mid x_1\geq x_2\}$.  Even more, $C_\Gamma=P$ because $(1:1)\in C_\Gamma$ (Proposition~\ref{P:Parabolic}), $(1:0)\in C_\Gamma$ (the second projection of the modular embedding contains an elliptic isometry of infinite order) and the two points are connected (Theorem~\ref{T:ConvexAlg}). Therefore their limit set is ``half'' of the geometric boundary of $(\bbH^2)^2$.
\end{rem}
 
\begin{proof}[Alternative proof of Theorem~\ref{T:BigLimitSet} in the case with parabolic isometries]
As in the proof of Theorem~\ref{T:BigLimitSet} it is enough to show that $P_\Gamma$ contains an $(r-1)$-dimensional ball.

Using Lemma~\ref{L:OneElliptic} we can construct elements $g_2,\ldots,g_r \in \Gamma$ such that $g_{ii}$ is elliptic for $i=2,\ldots,r$ and $g_{ik}$ is hyperbolic for $i\neq k$. Hence each translation directions $L(g_i)$, which is in the limit cone of $\Gamma$, has a coordinate $0$ at the $i$-th place. 

By Proposition~\ref{P:Parabolic}, $P_\Gamma$ and therefore the limit cone of $\Gamma$ contains the point $(1:\ldots:1)$. We can not represent $(1:\ldots:1)$ as a linear combination of $L(g_2),\ldots,L(g_r)$ because, as remarked before, for $L(g_i)=(x_{i1}:\ldots:x_{ir})$ we have $x_{i1}>x_{ij}$ for all $j=2,\ldots,r$ and $i=2,\ldots,r$. Therefore the limit cone of $\Gamma$, which is convex by Lemma~\ref{T:ConvexAlg}, contains the convex hull of $r$ linearly independent points and hence its interior $P_\Gamma$ contains an $(r-1)$-dimensional ball.
\end{proof}
\begin{rem}
If $\Gamma$ does not contain parabolic elements, then from the above proof it follows that $\Lim_\Gamma$ contains a subset homeomorphic to $D^{2r-2}$.
\end{rem}


\begin{thebibliography}{99}
\addcontentsline{toc}{section}{References}
\setcounter{tocdepth}{2}

\bibitem{aB95} A. Beardon,
	\emph{The Geometry of discrete groups}, Graduate Texts in Mathematics \textbf{91},
	Springer-Verlag, New York, 1995.

\bibitem{yB97} Y. Benoist, 
	\emph{Propri\'et\'es asymptotiques des groupes lin\'eaires}, 
  Geom. Funct. Anal. \textbf{7} (1997), 1-47.

\bibitem{aB81} A. Borel,
	\emph{Commensurability classes and volumes of hyperbolic 3-manifolds}
  Ann. Scuola Norm. Sup. Pisa Cl. Sci. (4) \textbf{8} (1981), pp. 1-33.


\bibitem{pC90} P. Cohen and J. Wolfart, 
	\emph{Modular embeddings for some non-arithmetic Fuchsian groups}, 
  Acta Arith. \textbf{61} (1990), 93-110.


\bibitem{fD99} F. Dal'Bo,
  \emph{Remarques sur le spectre des longueurs d'une surface et comptages},
  Bol. Soc. Bras. Math. Vol. \textbf{30}, no. 2, (1999), 199-221.

\bibitem{fD00} F. Dal'Bo and I. Kim,
  \emph{A criterion of conjugacy for Zariski dense subgroups},
  C. R. Acad. Sci. Paris, t. 330 s\'erie I (2000), 647-650.
  
\bibitem{pE96} P. Eberlein,
	\emph{Geometry of nonpositively curved manifolds},
	Chicago Lectures in Math., University of Chicago Press, Chicago and London, 1996.

\bibitem{sG09} S. Geninska,
	\emph{The limit set of subgroups of a class of arithmetic groups},\\
  http://digbib.ubka.uni-karlsruhe.de/volltexte/1000014024,
  Dissertation, Karlsruhe, 2009.

\bibitem{sG10a} S. Geninska,
	\emph{The limit set of subgroups of arithmetic groups in $\PSL(2,\C)^q\times\PSL(2,\R)^r$},
	arXiv:1001.1720.

\bibitem{sG08} S. Geninska and E. Leuzinger,
  \emph{A geometric characterization of arithmetic Fuchsian groups},
  Duke Math. J. \textbf{142} (2008), 111-125.

\bibitem{sK92} S. Katok,
	\emph{Fuchsian Groups},
	Chicago Lectures in Math., University of Chicago Press, Chicago, 1992.
	
\bibitem{gH62} G. Hardy and E. Wright,
	\emph{An introduction to the theory of numbers},
	Oxford University Press, London, fourth edition, 1962.
	

\bibitem{eL92} E. Leuzinger,
  \emph{On the trigonometry of symmentric spaces},
  Comment. Math. Helvetici \textbf{67} (1992), 252-286.

\bibitem{gL02} G. Link,
	\emph{Limit Sets of Discrete Groups acting on Symmetric Spaces},
	http://digbib.ubka.uni-karlsruhe.de/volltexte/7212002, Dissertation, Karlsruhe, 2002.

\bibitem{gL06} G. Link,
  \emph{Geometry and Dynamics of Discrete Isometry Groups of Higher Rank Symmetric Spaces}, 
  Geometriae Dedicata vol 122, no 1 (2006), 51-75.

\bibitem{cM03} C. Maclachlan and A. Reid,
	\emph{The Arithmetic of Hyperbolic 3-Manifolds}, 
	Graduate Texts in Mathematics \textbf{219},	Springer-Verlag, 2003.

\bibitem{gM91} G. A. Margulis, 
	\emph{Discrete subgroups of semisimple Lie groups}, 
	Springer-Verlag, Berlin, 1991.
	



\bibitem{pS96} P. Schmutz,
	\emph{Arithmetic groups and the length spectrum of Riemann surfaces},
	Duke Math. J. \textbf{84} (1996), 199-215.


\bibitem{pS00} P. Schmutz Schaller and J. Wolfart,
  \emph{Semi-arithmetic Fuchsian groups and modular embeddings},
  J. London Math. Soc. (2) \textbf{61} (2000), 13-24.
	

\bibitem{jS76} J. H. Sampson, 
	\emph{Sous-groupes conjugu\'es d'un groupe lin\'eaire}, 
	Ann. Inst. Fourier \textbf{26} (2)(1976), 1-6.

\bibitem{jS93} J. Stillwell,
	\emph{Classical Topology And Combinatorial Group Theory}, 
	no. 72 in Graduate Texts in Math., Springer-Verlag, second ed., 1993.

\bibitem{kT75} K. Takeuchi,
	\emph{A characterization of arithmetic Fuchsian groups},
	J. Math. Soc. Japan \textbf{27} (1975), 600-612.
	
\bibitem{kT77} K. Takeuchi,
	\emph{Arithmetic triangle groups},
	J. Math. Soc. Japan \textbf{29} (1977), 91-106.
%
\end{thebibliography}
\end{document}